\documentclass{article}
\usepackage{amsmath,amssymb,latexsym,amsfonts,amsthm}
\usepackage{epsfig,color,geometry,graphicx}
\usepackage{float,bbm}
\usepackage{epstopdf}
\numberwithin{equation}{section} % Include section number un the equations
\newcommand{\curl}{\mathop{\mathbf{curl}}\nolimits}
\newcommand{\dive}{\mathop{\mathrm{div}}\nolimits}
\newcommand{\grad}{\mathop{\mathbf{\nabla}}\nolimits}

\newcommand{\norm}[1]{\left\lVert#1\right\rVert}
\newcommand{\set}[1]{\left\{#1\right\}}
\newcommand{\produ}[1]{\left\langle#1\right\rangle}
\newcommand{\inner}[1]{\left(#1\right)}
\newcommand{\M}{M}
\newcommand{\X}{X}
\newcommand{\Y}{Y}
\newcommand{\R}{\mathbb{R}}
\newcommand{\oc}{\Omega_{{}_{\mathrm C}\!}}
\newcommand{\od}{\Omega_{{}_{\mathrm D}\!}}
\newcommand{\mo}{M(\od)}
\renewcommand{\H}{\mathrm{H}}
\renewcommand{\L}{\mathrm{L}}
\newcommand{\hcurlo}{\mathbf{H}(\curl;\Omega)}
\newcommand{\hocurlod}{\mathbf{H}_0(\curl;\od)}
\newcommand{\hcurloc}{\mathbf{H}(\curl;\oc)}
\newcommand{\hocurlo}{\mathbf{H}_0(\curl;\Omega)}
\newcommand{\hcurlod}{\mathbf{H}(\curl;\od)}
\newcommand{\hunod}{\H^1(\od)}
\newcommand{\ldoso}{\L^2(\Omega)}
\newcommand{\ldosod}{\L^2(\od)}
\newcommand{\ldoshocurlo}{\L^2(0,T;\hocurlo)}
\newcommand{\ldosmo}{\L^2(0,T;\mo)}

\newcommand{\cero}{\boldsymbol{0}}
\newcommand{\ds}{\,ds}
\newcommand{\dt}{\,dt}
\newcommand{\bn}{\mathbf{n}}
\newcommand{\xn}{\mathbf{x}}
\newcommand{\zn}{\mathbf{z}}
\newcommand{\un}{\mathbf{u}}
\newcommand{\fn}{\mathbf{f}}
\newcommand{\vn}{\mathbf{v}}
\newcommand{\E}{\mathcal{E}}

\newcommand{\wn}{\mathbf{w}}
\newcommand{\Jn}{\mathbf{J}}
\newcommand{\Hn}{\mathbf{H}}

\newcommand{\bJ}{\mathbf J}
\newcommand{\bE}{\mathbf E}

\newcommand{\bH}{\mathbf H}
\newcommand{\pct}{\mathrm{a.e.}}
\newcommand{\gt}{\boldsymbol{\gamma}_{\tau}}
\newcommand{\Vd}{V_{{}_{\mathrm D}\!}}
\newcommand{\vc}{\vn_{{}_{\!\mathrm C}\!}}
\newcommand{\vd}{\vn_{{}_{\!\mathrm D}\!}}

\newcommand{\VV}{V}

\newcommand{\hounoO}{\H_0^1(\Omega)}
\newcommand{\lodoso}{\L_0^2(\Omega)}
\newcommand{\ldoshounod}{\L^2(0,T;\hounoO^d)}
\newcommand{\ldoslodoso}{\L^2(0,T;\lodoso)}
\newcommand{\hunolodoso}{\H^1(0,T;\lodoso)}
\newcommand{\hmunoo}{\H^{-1}(\Omega)}
\newcommand{\ldosldoso}{\L^2(0,T;\ldoso)}

\newcommand{\houno}{\H_0^1(\Omega)}
\newcommand{\Gc}{\Gamma_{{}_{\!\mathrm C}\!}}
\newcommand{\Gd}{\Gamma_{{}_{\!\mathrm D}\!}}
\newcommand{\GI}{\Gamma_{{}_{\!\mathrm I}\!}}
\newcommand{\GE}{\Gamma_{{}_{\!\mathrm E}\!}}
\newcommand{\GJ}{\Gamma_{{}_{\!\mathrm J}\!}}

\newcommand{\VVd}{V_{{}_{\mathrm D}\!}}
\newcommand{\nn}{\boldsymbol{n}}
\newcommand{\HGccurloc}{\mathbf{H}_{\Gc}(\curl;\oc)}
\newcommand{\HGdcurlod}{\mathbf{H}_{\Gd}(\curl;\od)}
\newcommand{\HGIcurlod}{\mathbf{H}_{\GI}(\curl;\od)}
\newcommand{\HHGIcurlod}{\widehat{\mathbf{H}}_{\GI}(\curl;\od)}
\newcommand{\HGIcurlood}{{\mathbf{H}}_{\GI}(\curl^{\cero};\od)}
\newcommand{\HGddivood}{{\mathbf{H}}_{\Gd}(\dive^{0}_{\varepsilon};\od)}
\newcommand{\hdivod}{\mathrm{H}(\dive,\od)}
\newcommand{\HHdivod}{\widehat{\mathrm{H}}(\dive;\od)}
\newcommand{\LdosGd}{\mathrm{L}^2(\Gd)}
\newcommand{\HH}{\mathbb{H}}
\newcommand{\wnd}{\wn_{{}_{\mathrm D}\!}}
\newcommand{\vnd}{\vn_{{}_{\mathrm D}\!}}
\newcommand{\znd}{\zn_{{}_{\mathrm D}\!}}
\newcommand{\gtc}{\boldsymbol{\gamma}_\tau^{{}^{\mathrm C}\!}}
\newcommand{\gtd}{\boldsymbol{\gamma}_\tau^{{}^{\mathrm D}\!}}
\newcommand{\rangoGtoc}{\mathrm{H}^{-1/2}(\dive_{\tau};\partial\oc)}
\newcommand{\rangoGtod}{\mathrm{H}^{-1/2}(\dive_{\tau};\partial\od)}
\newcommand{\etan}{\boldsymbol{\eta}}
\newcommand{\LL}{\mathcal{L}}

\newcommand{\hdive}{\mathbf{H}(\dive;\Omega)}
\newcommand{\hodive}{\mathbf{H}_0(\dive;\Omega)}
\newtheorem{theorem}{Theorem}[section]

\newtheorem{lemma}{Lemma}[section]
\newtheorem{corollary}{Corollary}[section]

\newtheorem{problem}{Problem}[section]
\newtheorem{remark}[theorem]{Remark}
\def\thanksRamiroChristian{Universidad del Cauca, Popay\'an, Colombia, 
 email: {\tt rmacevedo@unicauca.edu.co}, {\tt christiancamilo@unicauca.edu.co}}
                  
\def\thanksBibiana{Universidad Nacional sede Medell\'in, Colombia
email: {\tt blopezr@unal.edu.co}.}  
\begin{document}
%%%---------------------------------------------------------------------
\title{
Well-posedness for a family of degenerate 
parabolic mixed equations}


\author{
{\sc Ramiro Acevedo}\thanks{\thanksRamiroChristian}\quad
{\sc Christian G\'omez}$^{*,\dag}$\quad  
{\sc Bibiana L\'opez-Rodr\'{\i}guez}\thanks{\thanksBibiana}
}
\date{11 Mayo 2020}

\maketitle

%-----------------------------------------------------------------------
\begin{abstract}
The aim of this work is to show an abstract framework to analyze a family of linear degenerate parabolic mixed equations. We combine the theory for the degenerate parabolic equations (see, \textit{e.g.}, \cite{showalter}) with the classical Babu$\check{\text{s}}$ka-Brezzi theory for linear mixed stationary equations to deduce sufficient conditions to prove the well-posedness of the problem. Finally, we illustrate the application of the abstract framework through examples that come from physical science applications including fluid dynamics models and electromagnetic problems. 
\end{abstract}

\textbf{keyword}
Well-posedness, time dependent problems, parabolic degenerate equations, mixed equations, Stokes problem, eddy current model. 

%\MSC[2010] 78M10\sep 65M60

%\end{frontmatter}

%\linenumbers

%---------------------------------------------------------------------
%\section{Introduction}
\section{Introduction}
{This work is concerned with an abstract theory to study the existence and uniqueness of solution of a family of linear mixed degenerate evolutions problems. 
This kind of degenerate {system arises} in many applications to coupled multi-physics models, for instance in electromagnetic applications where it is necessary to study the problem in two types of regions (the conductor and the insulator) and the electromagnetic fields satisfy some divergence-free or curl-free condition. 
Furthermore, the study of these equations are relevant for the numerical analysis in {finite-dimensional} approximations, \textit{v.g.}, approximations that correspond to finite element methods. }

The idea of combining  mixed variational formulation with time dependent evolution problem is not new, starting from the 1980s, for instance, \cite{JT,BR}, and more recently, to study a dynamics fluid problem \cite{BG} and electromagnetic applications \cite{A1,A2,blrs:13,MS1}.
\textit{Bernardi} \& \textit{Raugel} \cite{BR} have introduced some abstract framework for an usual (non-degenerate) mixed parabolic equation inspired on 
a dynamics fluid model called the Stokes problem. More precisely, they analyze the existence and uniqueness of solution of the following mixed parabolic problem: 

\noindent find $u\in \L^2(0,T;\X)$ and $\lambda\in \mathcal{D}'(0,T;\M)$ such that: 
\begin{align}
&\dfrac{d}{dt} {(u(t),v)_{\Y}} + b(v,\lambda(t)) + a(u(t),v) = \left\langle f(t),v\right\rangle_{\X}&& \forall v\in \X,\nonumber\\
&b(u(t),\mu)=0&& \forall \mu\in \M,\label{abstracto-BR}\\
&u(0)=u_0,\nonumber
\end{align}
where $\X$, $\Y$ and $\M$ are real Hilbert spaces with the embedding ${\X\subseteq \Y}$ is continuous and dense, $\mathcal{D}'(0,T;\M)$ is the space of $\M$-valued distributions on $[0,T]$, $(\cdot,\cdot)_{\Y}$ is the inner product on $\Y$, $a:\X\times \X\to \R$, $b:\X\times \M\to \R$ are continuous bilinear forms, $u_0\in\Y$ and $f\in\L^2(0,T;\X')$.

However, it is not possible to apply this abstract theory to the formulations arising from electromagnetic problems studied in \cite{A1,A2,blrs:13}, because in these cases the first term inside of the time-derivative is not an inner product in the whole space $Y$, namely that problems are degenerate.  Similarly, that abstract theory can not be applied to the formulation analyzed in \cite{MS1}, because in this formulation the right-hand term of the second equation in \eqref{abstracto-BR} is non-zero. 

This paper is devoted to analyze a new abstract framework for a family of mixed parabolic equations including both degenerate and non-homogeneous cases as the problems mentioned above, which leads to a more general problem of  \eqref{abstracto-BR}. We combine the linear degenerate parabolic equation theory and the classical Babu$\check{\text{s}}$ka-Brezzi Theory to prove the existence and uniqueness of solution, by assuming some reasonable conditions inspired by the application problems. More precisely, {the} first step is to apply the theory for degenerate parabolic problems to the reduced equation to the kernel of the bilinear form given in the second equation, where {the bilinear form of the first equation must verify} a G\r{a}rding-type inequality. After that, we use {an} inf-sup condition for the bilinear form in the second equation, to prove the existence of the Lagrange multiplier of the problem. The uniqueness of solution is obtained from the time regularity of the Lagrange Multiplier.

The application problems of the theory are organized depending on their physical origin. Firstly, we show that a fluid dynamics model (the classical linear time-dependent Stokes problem) can be studied as a particular case of the theory. This example proves that the usual (non-degenerate) mixed parabolic equations can be analyzed with our theory.
The second set of applications comes from an electromagnetic model called \textit{eddy current model}, which is obtained from Maxwell equations by assuming that the current displacement can be dropped in the Law of Ampere-Maxwell. The two considered variational formulations have a main unknown a time-primitive of the electric field, but in the first case{,} the conducting domain is compactly included in the computational domain, while in the second one, since a more general case, realistic boundary conditions can be considered. The well-posedness of this last model for the eddy current problem was studied in \cite{blrs:13} by showing that the model is equivalent to the other system of equations
{in which} the existence and uniqueness of solution is previously known. In contrast, our theory allows us to deduce directly the well-posedness of the model without the need to resort to an equivalent problem.  

Other studies {about abstract framework} theories have been proposed in recent years for evolution mixed problems. For instance, in \cite{showalterpaper} the problem has been studied by using the semigroups theory with application to the time-dependent Stokes flow, {and the analysis} for linear and nonlinear case{s} by taking a model problem in fluid-flow on poroelastic porous medium was presented in \cite{VS}.

The outline of the paper is as follows: Section \ref{continuo} is devoted to obtain the abstract framework for a family of mixed degenerate parabolic equations and its analysis of existence and uniqueness of solution. In Section \ref{aplicaciones}, we show the application of the theory to a dynamic fluids model {(the time-dependent Stokes problem) }and to electromagnetic problems (the eddy current model with internal conductor and with input current), where we use the abstract theory to deduce the well-posedness for both models. Finally, in Section~\ref{conclusiones} we show a brief list of conclusions of the work.

%---------------------------------------------------------------------
%\setcounter{equation}{0}
\section{An abstract degenerate mixed parabolic problem}
\label{continuo}

Let $\M$ be real reflexive Banach space and Let $\X$ and $\Y$ be two real Hilbert spaces such that $\X$ is contained in $\Y$ with a continuous and dense embedding. Let $a:\X\times \X\to \R$ and $b:\X\times \M\to \R$ be two continuous bilinear forms and $R: \Y\to \Y'$ a continuous and linear operator. Let $A$ be the linear and continuous operator induced by the bilinear form $a$, \textit{i.e.}, $A:\X\to \X'$ given by
\begin{equation*}
\langle Au,w \rangle_{\X}:=a(u,w)\qquad \forall u,w\in\X.
\end{equation*}

Let $V$ be the kernel of the bilinear form $b$, \textit{i.e.},
\begin{equation*}
V:=\left\{v\in \X: b(v,\mu)=0\ \ \forall\mu\in \M\right\},
\end{equation*}
and denote by $W$ its clausure with respect to the $\Y$-norm, \textit{i.e.},
\begin{equation*}
W:=\overline{V}^{\| \cdot \|_{\Y}}.
\end{equation*}

We consider now the following problem, which is the main problem of this section: Given $u_0\in\Y$, $f\in\L^2(0,T;\X')$ and $g\in\L^2(0,T;\M')$, the continuous problem is 
\begin{problem}
\label{PC1}
Find $u\in\L^2(0,T;\X)$ and $\lambda\in\L^2(0,T;\M)$ satisfying the following equations:
\begin{align*}
& \frac{d}{dt}\left[\langle Ru(t),v\rangle_{\Y}+ b(v,\lambda(t))\right]+\langle Au(t),v\rangle_{\X}
=\langle f(t),v\rangle_{\X}&& \forall v\in\X \quad\text{in }\mathcal{D}'(0,T),\\
& b(u(t),\mu)=\langle g(t),\mu\rangle_{\M}&& \forall\mu\in\M,\\
& \langle Ru(0),v\rangle_{\Y}=\langle Ru_0,v\rangle_{\Y}&& \forall v\in \Y.
\end{align*}
\end{problem}

Next, in order to show that the previous problem is a well-posed problem, we additionally assume that the spaces, 
forms, operators and data of problem have the following properties:
\begin{itemize}
\item[H1.] The bilinear form $b$ satisfies a continuous \textit{inf--sup} condition, \textit{i.e.}, there exists $\beta>0$ such that
\begin{equation*}
\underset{v\in\X}{\sup}\frac{b(v,\mu)}{\|v\|_{\X}}\geq \beta \|\mu\|_{\M}\quad\forall \mu\in\M.
\end{equation*}
\item[H2.] $R$ is self-adjoint and monotone on $V$, \textit{i.e.}, 
\begin{equation*}
\langle Rv,w \rangle_{\Y}=\langle Rw,v \rangle_{\Y},\qquad \langle  Rv,v \rangle_{\Y}\geq 0\ \qquad \forall v,w\in V.
\end{equation*}
\item[H3.] The operator $A$ is self-adjoint on $V$, \textit{i.e.},
\begin{equation*}
\langle Av,w \rangle_{\X}=\langle Aw,v \rangle_{\X}\ \qquad \forall v,w\in V.
\end{equation*}
\item[H4.] There exist $\gamma>0$ and  $\alpha>0$ such that
\begin{equation}\label{garding1}
\langle Av,v \rangle_{\X}+\gamma\langle Rv,v \rangle_{\Y}\geq \alpha \|v\|^2_{\X}\ \qquad \forall v\in V.
\end{equation}
\item[H5.] The initial data $u_0$ belongs to $W$.
\item[H6.] The data function $g$ belongs to $\H^1(0,T;\M')$.
\end{itemize}

Now, we show the main result of this {paper}. From now on, we will denote by $C$ a generic constant which is not necessarily the same at each occurrence.
\begin{theorem}
\label{main}
Let us assume that assumptions $\mathrm{H}1-\mathrm{H}6$ hold true. Then the Problem~\ref{PC1} has a unique solution $(u,\lambda)$ %$u\in\L^2(0,T;\X)$ and 
with $\lambda\in \H^1(0,T;\M)$, and there exists a constant $C>0$ such that
\begin{equation*}
\|u\|_{\L^2(0,T;\X)} + \|\lambda\|_{\L^2(0,T;\M)} 
\leq C 
\left\{
\|f\|_{\L^2(0,T;\X')} + \|g\|_{\H^1(0,T;\M')} + \|u_0\|_{\Y}
\right\}.
\end{equation*}
Moreover, $\lambda(0)=0$.
\end{theorem}

\begin{proof} 
\textit{Existence.} We show that a solution $u$ for the Problem~\ref{PC1}  is given by
\begin{align}\label{solucion summa de u tilde y z}
u(t):=\tilde {u}(t)+z(t) \quad \text{a.e} \quad t\in [0,T],
\end{align}
where $\tilde{u}\in \L^2(0,T;V)$ and $z\in \L^2(0,T;V^{\bot})$.
In fact, let $B:X \to M'$ be the operator induced by the bilinear form $b$, \textit{i.e.}, 
\[
\langle Bv,\mu \rangle_{\M}:=b(v,\mu)\quad \forall v\in\X,\ \forall \mu\in\M.
\]
Since $b$ satisfies the \textit{inf--sup} condition, for each $t\in[0,T]$ there exists a unique $z(t)\in V^\perp$ such that 
\[
\langle Bz(t),\mu\rangle_{\M}=\langle g(t),\mu\rangle_{\M}\quad\forall \mu\in\M,
\]
and there exists a constant $C>0$ such that $\|z\|_{\L^2(0,T;\X)}\leq C \|g\|_{\L^2(0,T;\M')}$. 
Moreover, by recalling $g\in\H^1(0,T;\M')$ (see H6), it follows that $z\in\H^1(0,T;X)$, which implies
\[
\|z\|_{\H^1(0,T;\X)}\leq C \|g\|_{\H^1(0,T;\M')}.
\]

Next, we have to consider the solution of the following degenerate parabolic problem, 
\begin{problem}\label{parabolicoutilde}
 Find $\tilde{u}\in\mathrm{L}^2(0,T;V)$ such that:
\begin{align*}
&\frac{d}{dt}\langle R\tilde{u}(t),v\rangle_{\Y}
+\langle A\tilde{u}(t),v\rangle_{\X}
=\langle f(t),v\rangle_{\X}
-\frac{d}{dt}\langle Rz(t),v\rangle_{\Y}
-\langle Az(t),v\rangle_{\X}
&&\forall v\in V \quad\text{in }\mathcal{D}'(0,T)\\
&\langle R\tilde{u}(0),v\rangle_{\Y}
=\langle R(u_0-z(0)),v\rangle_{\Y} 
&&\forall v\in\Y.
\end{align*}
\end{problem}
We can notice that $z(0)$ can be computed because $z\in \H^1(0,T;X)$.
It is straightforward to verify that the Problem \ref{parabolicoutilde} has a unique solution, 
which satisfies the following inequality (see, for instance, \cite[Chapter~3, Propositions 3.2 and 3.3]{showalter})
\[
\|\tilde u\|_{\L^2(0,T;V)} \leq C \left\{\left\|f-Az-\frac{d}{dt}Rz\right\|^2_{\L^2(0,T;V')} + \langle R(u_0-z(0)),u_0-z(0)\rangle_{Y}\right\}^{1/2},
\]
hence,
\[
\|\tilde{u}\|_{\L^2(0,T;\X)}\leq C
\left\{ 
\|f\|_{\L^2(0,T;\X')} + \|g\|_{\H^1(0,T;\M')} + \|u_0\|_{\Y}
\right\}.
\]
Now, we define $u$ as in \eqref{solucion summa de u tilde y z}. Then $u$ satisfies the second and third equation the Problem~\ref{PC1} and the following inequality holds
\[
\|u\|_{\L^2(0,T;\X)} \leq C 
\left\{
\|f\|_{\L^2(0,T;V')} + \|g\|_{\H^1(0,T;\M')} +\|u_0\|_{Y}
\right\}.
\]

Next, in order to show the existence of the Lagrange multiplier $\lambda$, we consider the operator $G\in L^{2}(0,T;\X')$  defined by
\begin{equation}\label{defG}
\int_0^T \langle G(t),v(t) \rangle_{\X} \dt
:= \int_0^T \left[
\langle R(u_0-u(t)),v(t) \rangle_{\Y}
-\int_0^t \langle Au(s),v(t)\rangle_{\X} \ds 
+\int_0^t  \langle f(s),v(t)\rangle_{\X} \ds
\right]\dt
\end{equation}
for all $v\in \L^2(0,T;\X)$. Note that $G\in\H^1(0,T;\X')$, in fact, $\dfrac{d}{dt} G:[0,T]\to\X'$ is given by
\[
\left\langle \frac{d}{dt} G(t),v\right\rangle_{\X}:= \frac{d}{dt}\langle R(u_0-u(t)),v\rangle_{\Y}-\langle Au(t),v\rangle_{\X}-\langle f(t),v\rangle_{\X}\qquad \forall v\in \X
\]
and therefore $G\in\mathcal{C}([0,T];\X)$. We need to prove that $G\in \L^2(0,T;V)^{\circ}$, \textit{i.e.},
\begin{equation}\label{anhilador de G2}
\int_0^T \langle G(t),v(t)\rangle_{\X} \dt=0 \quad \forall v\in \L^2(0,T;V).
\end{equation}
In fact, it is easily seen  that the Problem \ref{parabolicoutilde} is equivalent to finding  $\tilde{u}\in \L^2(0,T;\X)$
such that for all $v\in \L^2(0,T;V)\cap\H^1(0,T;W)$ with $v(T)=0$, the following identity holds:
\begin{align*}
&-\int_0^T \langle R\tilde u(t),v'(t)\rangle_{\Y}\dt+\int_{0}^{T}\langle A\tilde u(t),v(t)\rangle_{\X}\,dt\\
&\qquad=\int_{0}^{T}\langle f(t),v(t)\rangle_{\X}\dt
-\int_{0}^{T}\langle Az(t),v(t)\rangle_{\X}\dt
+\int_{0}^{T}\langle Rz(t),v'(t)\rangle_{\Y}\dt
+\langle R(\tilde u(0)+z(0)),v(0)\rangle_{\Y}.
\end{align*}
By testing with $v:=\phi\in \mathrm{C}^{\infty}(0,T;V)$ given by 
$\phi(t):= \int_{t}^{T}\xi(s)\ds$ where $\xi\in \mathrm{C}_0^{\infty}(0,T;V)$, and using a variable change, we get
\begin{align*}
&\int_0^T \langle R \tilde u(t), \xi(t)\rangle_{\Y} \dt
+\int_0^T\int_0^t{\langle A\tilde u(s),\xi(t)\rangle_{\X}}\ds\dt  \\
&\ \ =\int_0^T\int_0^t \langle f(s),\xi(t)\rangle_{\X}\ds\dt 
-\int_0^T\int_0^t \langle Az(s),\xi(t)\rangle_{\X}\ds\dt
+\int_0^T\langle R(\tilde u(0)+z(0))-Rz(t),\xi(t)\rangle_{Y}\dt 
\end{align*}
or equivalently,
\[
\int_0^T \langle G(t),\xi(t)\rangle_{\X}\dt=0 \quad \forall\xi\in \mathrm{C}_0^{\infty}(0,T;V).
\]
Therefore, by using the density of $\mathrm{C}_0^{\infty}(0,T;V)$ in $\L^2(0,T;V)$, \eqref{anhilador de G2} follows. 
Thus, the \textit{inf--sup} condition implies there exists a unique
$\lambda\in \L^2(0,T;\M)$ such that
\begin{align}\label{igualdad de lambda2}
\int_0^T b(v(t),\lambda(t))\dt
=\int_0^T \langle G(t),v(t)\rangle_{\X}\dt \quad \forall v\in \L^2(0,T;\X),
\end{align}
and satisfying
\[
\|\lambda\|_{\L^2(0,T;\M)}\leq C 
\left\{
\|f\|_{\L^2(0,T;\X')}+\|g\|_{\H^1(0,T;\M')}+\|u_0\|_{Y}
\right\}.
\]
{It is easy prove $\lambda\in \H^1(0,T;\M)$ because $G\in\H^1(0,T;\X')$ and the \textit{inf--sup} condition. Now, we will prove} that $u$ and $\lambda$ verify \eqref{PC1}. In fact, let $v\in X$ and 
$\phi\in \mathrm{C}_0^{\infty}(0,T)$. Let $\xi\in \L^2(0,T;\X)$ given by
\[
t\mapsto \xi(t):=-v\phi'(t).
\]
By testing \eqref{igualdad de lambda2} with $\xi$ and using integration by parts, we obtain
\begin{equation*}
\int_0^T\langle G(t),\xi(t)\rangle_{\X}\dt
=\int_0^T b(\xi(t),\lambda(t))\dt
=\int_0^T \phi(t)\frac{d}{dt}b(v,\lambda(t))\dt
\end{equation*}
and by recalling the definition of $G$ (see \eqref{defG})
\begin{equation*}
\int_0^T\langle G(t),\xi(t)\rangle_{\X}\dt
=\int_0^T \phi(t)\left\{
-\frac{d}{dt}\langle Ru(t),v \rangle_{\Y}
-\langle Au(t),v\rangle_{\X} 
+\langle f(t),v\rangle_{\X}\right\}\dt.
\end{equation*}
Therefore,
\[
\int_0^T\phi(t) \left\{
\frac{d}{dt}[\langle Ru(t),v\rangle_{\Y}
+b(v,\lambda(t))] 
+\langle Au(t),v\rangle_{\X} 
-\langle f(t),v\rangle_{\X}\right\}\dt=0
\]
for all $\phi\in \mathrm{C}_0^{\infty}(0,T)$ and $v\in X$, and consequently $u$ and $\lambda$ satisfy \eqref{PC1}.

\textit{Uniqueness.}  Let $(u,\lambda)$ be a solution of the Problem~\ref{PC1} with $u_0=0$, $f=0$ and $g=0$. We need to prove that
 $(u,\lambda)=(0,0)$ by assuming that $\lambda\in\H^1(0,T;\M)$.
In fact, we can notice that from first identity the Problem~\ref{PC1} it follows $u\in\L^2(0,T;V)$. Consequently, by testing 
the Problem~\ref{PC1} with $v\in V$, we deduce that $u$ is a solution of the homogeneous degenerate parabolic problem in the kernel $V$:
\begin{align*}
&\frac{d}{dt}\langle R{u}(t),v\rangle_{\Y}
+\langle A{u}(t),v\rangle_{\X}
=0
&&\forall v\in V \quad\text{in }\mathcal{D}'(0,T),\\
&\langle R{u}(0),v\rangle_{\Y}=0 
&&\forall v\in\Y.
\end{align*}
This Problem has at most one solution by virtue of \cite[Chapter III, Proposition~3.3]{showalter} and therefore $u=0$. Then,  we have
{
\[b(v,\partial_t\lambda(t))=0 \quad \forall v\in X\quad \pct \ t\in(0,T),\]
}
thus, the \textit{inf--sup} condition of $b$ yields {$\partial_t\lambda=0$, then $\lambda$ is a time independent variable. 
Furthermore, we can deduce that $\lambda(0)=0$. In fact, by testing \eqref{igualdad de lambda2} with $v(t)=w \varphi(t)$, $w\in\X$, $\varphi\in\mathrm{C}_0^{\infty}(0,T)$, it follows that
\[
\langle G(\cdot),w\rangle_{\X}=b(w,\lambda(\cdot))\quad \forall w\in \X\quad \text{in }\mathcal{C}([0,T]).
\]
Now, by recalling that $\lambda\in \mathcal{C}([0,T];\M)$, the continuous \textit{inf--sup} condition gives
\begin{equation*}%\label{lambdaCero}
\beta\|\lambda(0)\|_{\M}
\leq \underset{\underset{w\neq 0}{w\in\X}}{\sup}\frac{b(w,\lambda(0))}{\|w\|_{\X}}
=\underset{\underset{w\neq 0}{w\in\X}}{\sup}\frac{\langle G(0),w\rangle_{\X}}{\|w\|_{\X}}=0,
\end{equation*}
which implies $\lambda(0)=0$, thus we deduce that $\lambda=0$.
}
\end{proof}
%--------------------------------------------------------------------
\section{Applications}\label{aplicaciones} %to Mechanic of fluids and eddy current model}\label{aplicaciones}
%--------------------------------------------------------------------
\subsection{Applications to fluid dynamics: The time-dependent Stokes problem.}
%--------------------------------------------------------------------
The \textit{time-dependent Stokes equation system} is a fundamental model of viscous flow because it represents the asymptotic limiting form of the \textit{Navier-Stokes problem} when the Reynolds number becomes very small
\cite[Chapter~14]{panton}. In this limit, the fluid dynamic is mainly controlled by diffusion and the non-linear convection term in the full Navier-Stokes equation can be dropped to obtain the so-called \textit{Stokes problem}. Stokes flows are important in lubrication theory, in porous media flow and in certain biological applications namely in the swimming of microorganisms, in microfluidics applications and in the flow of blood in parts of the human body. 

In order to put everything into a mathematical framework, let $\Omega$ be an open, bounded and connected  subset of $\R^d$ the domain occupied by the fluid, being $d$ either $2$ or $3$ the space dimension. The boundary of $\Omega$ is denoted by $\Gamma:=\partial\Omega$ and assumed to be Lipschitz continuous. Then, the strong time-dependent Stokes consists 
\begin{problem}
\label{problemastokes}
Find $\un:\Omega \times [0,T]\rightarrow \R^d$ and $p: \Omega\times [0,T]\rightarrow \R$ such that
\begin{align*}
\dfrac{\partial\un}{\partial t} 
-\nu \Delta\un + {\nabla}p 
= \boldsymbol{f}&\qquad \text{in}\ \Omega\times [0,T],
\\
\dive\un = 0 &\qquad \text{in}\ \Omega\times [0,T],
\\
\un= \boldsymbol{0}&\qquad \text{on}\ \Gamma\times[0,T], 
\\
\un(\cdot,0)=\un_0(\cdot) &\qquad \text{in}\ \Omega.
\end{align*}
\end{problem}

The variable $\un$ is a vector-valued function representing the velocity of the fluid, and the scalar function $p$ represents the pressure. First equation, means conservation of the momentum of the fluid (and so is the \textit{momentum equation}). The second equation enforces conservation of mass and in the specialized literature it is also referred as the \textit{incompressibility constraint}. In third equation, we are considering the classical homogeneous Dirichlet boundary condition, but other condition involving the normal derivative of $\un$ or a linear combination between the latter and $p$ on $\Gamma$ can also be considered in the model, see, for instance, \cite[Section~10.1.1]{quarteronivalli}. The last equation is the initial condition given by the known data function $\un_0$ which is the initial velocity. Other data of the problem are the positive constant $\nu$ and a given vector function $\fn$ which is the body force acting on the fluid.

We can observe that if $(\un,p)$ is a solution of the Problem \ref{problemastokes} then if a constant $\kappa$ is added to the pressure field solution $p$, the pair $(\un,p+\kappa)$ {is another solution.}
Consequently, the pressure can be restricted to have zero-mean over $\Omega$, \textit{i.e.}, $p(t)\in\lodoso\quad \forall t\in[0,T]$,  where 
\begin{equation*}
\lodoso:=\left\{
q\in\ldoso:\ \int_{\Omega}q=0
\right\},
\end{equation*}
endowed with the usual norm in $\ldoso$. 

More precisely, according to \cite{BR}, the pressure belongs to the space of distributions with variable in $[0,T]$ and value in $\lodoso$, \textit{i.e.}, $p\in \mathcal{D}([0,T];\lodoso)$. Furthermore, the variational formulation for the Stokes Problem proposed in \cite{BR} can read as follows:
\begin{problem} %\label{variationalstokes}
Find $\un\in\ldoshounod$ and $p\in\mathcal{D}([0,T];\lodoso)$
such that
\begin{align*}
&\frac{d}{dt}\left( \int_{\Omega} \un(t)\cdot \vn \right)+ \nu \int_{\Omega }\nabla\un(t) : \nabla\vn
-\int_\Omega p(t)\dive \vn=\int_{\Omega}\boldsymbol{f}(t)\cdot\vn&&\forall\vn\in\hounoO^d,\\
&\int_\Omega q\dive\un = 0&&\forall q\in\lodoso,\\
&\un(\cdot,0)=\un_0(\cdot) \quad \text{in}\ \Omega,
\end{align*}
\end{problem}
Here, the tensor product $\zn:\wn$ is given by $\zn:\wn:=\sum_{i=1}^d\sum_{j=1}^d \zn_{ij}\wn_{ij}$ for all $\zn,\wn\in \L^2(\Omega)^{d\times d} $. Next, in order to obtain an equivalent formulation to the previous problem but with the structure of the family of problems studied in Section~\ref{continuo}, 
we introduce the time-primitive of the pressure:
\begin{equation*}
P(\xn,t):=\int_{0}^tp(\xn,s)ds,\qquad \xn\in\Omega,\quad t\in[0,T].
\end{equation*}
Consequently, we obtain the following variational formulation for the Stokes problem: 
\begin{problem}\label{stokesvariational2}
Find $\un\in\ldoshounod$ and $P\in\ldoslodoso$ such that
\begin{align*}
&\frac{d}{dt}\left(\int_{\Omega}\un(t)\cdot \vn -\int_\Omega P(t)\dive \vn\right)
+ \nu \int_{\Omega }\nabla\un(t) : \nabla\vn
=\int_{\Omega}\boldsymbol{f}(t)\cdot\vn&& \forall\vn\in\hounoO^d,
\\
&\int_\Omega q\dive\un = 0&&\forall q\in\lodoso,
\\
&\un(\cdot,0)=\un_0(\cdot) \quad \text{in}\ \Omega.
\end{align*}
\end{problem}
%--------------------------------------------------------------------
\subsubsection{Well-posedness of the continuous mixed variational Stokes problem}
Before starting our analysis for the well-posedness of Problem \ref{stokesvariational2}, we need to recall some properties of gradient and divergence operators.
We first recall that the gradient operator  
\[
\nabla:\ldoso\to\hmunoo^d,
\]
is defined by 
\[
\produ{\nabla q,\vn}_{\houno^d}:=\int_{\Omega} q\dive\vn\qquad\forall q\in \ldoso,\quad\forall \vn\in \houno^d,
\]
where $\produ{\cdot,\cdot}_{\houno^d}$ denotes the duality pairing between $\hmunoo^d$ and $\houno^d$.
Now, we denote by $V$ the kernel of the divergence operator, \textit{i.e.},
\begin{equation}\label{defV-S}
V:=\set{\vn\in\houno^d:\ \dive\vn=0}.
\end{equation}
Let $V^\bot$ and $V^\circ$ the orthogonal and the annihilator of $V$ respectively, \textit{i.e.},
\begin{align}
V^\bot&:=\set{\vn\in\houno^d:\ \inner{\vn,\wn}_{\houno^d}=0\quad\forall\wn\in V},\nonumber
\\
V^\circ&:=\set{\eta\in\hmunoo^d:\ \produ{\eta,\vn}_{\houno^d}=0\quad\forall\vn\in V}.\nonumber
\end{align}

The following result shows the importance by these previous subspaces in the analysis of the operators $\nabla$ and $\dive$. Moreover, it proves that the range space of $\dive$ is exactly $\lodoso$.
\begin{lemma}\label{isomorfismo-div-grad}
Let $\Omega$ be connected.
\begin{enumerate} 
\item The operator $\nabla$ is an isomorphism of $\lodoso$ onto $V^\circ$.
\item The operator $\dive$ is an isomorphism of $V^\bot$ onto $\lodoso$.
\end{enumerate}
\end{lemma}
\begin{proof}
See, for instance, \cite[Corollary~I.2.4]{GR}.
\end{proof}

Next, we deduce an immediate consequence of this previous lemma, which is important to prove the well-posedness of the variational formulation of the Stokes problem.
\begin{corollary}\label{div-inverso}
There exists a constant $C>0$ such that
\begin{equation*}
\norm{\vn}_{\houno^d}\leq C\norm{\dive\vn}_{\ldoso}\qquad\forall\vn\in V^\bot.
\end{equation*}
\end{corollary}
\begin{proof}
The linear operator $\dive:\houno^d\to\lodoso$ is clearly continuous. Furthermore, since $V^\bot$ is a closed subspace of $\houno^d$, $V^\bot$ is a Banach space and therefore, from the previous lemma, $\dive$ is a linear, continuous and bijective operator between Banach spaces. Hence, its inverse operator is also continuous and the result follows. 
\end{proof}

Now, we wish to analyze the Problem \ref{stokesvariational2}  by using the abstract theory studied in Section~\ref{continuo}. To this aim, we start by denoting
\begin{equation*}
X:=\hounoO^d,\qquad Y:=\ldoso^d,\qquad M:=\lodoso,
\end{equation*}
with their usual inner products. Actually, by using the basic properties of Sobolev spaces, it is a simple matter to deduce that these spaces satisfying the required properties for the theory. Moreover, we need to define the bilinear forms
{ $a:X\times X\to\R$ and $b:X\times M\to\R$ given by:
\begin{align*}
a(\vn,\wn)&:=\nu\int_{\Omega }\nabla\vn : \nabla\wn\qquad\forall\vn,\wn\in X,\\
b(\vn,q)&:= -\int_\Omega q\dive \vn\qquad\forall\vn\in X,\quad\forall q\in M.
\end{align*}}
Let us notice that the space $V$ defined in \eqref{defV-S} is precisely the kernel of the bilinear form $b$, \textit{i.e.},
\[
V=\left\{\vn\in \X: \ b(\vn,q)=0\quad \forall q\in \M\right\}
=\set{\vn\in\houno^d:\ \int_{\Omega}q\dive\vn = 0\quad\forall q\in\lodoso}.
\]
We can now state and prove the well-posedness for the weak formulation of the time-dependent Stokes problem. 
\begin{theorem}
Let $\un_0\in\ldoso$ and $\fn\in\ldosldoso$. The Problem \ref{stokesvariational2} has a unique solution $(\un,P)\in \ldoshounod\times\hunolodoso$ satisfying
\begin{equation*}
\norm{\un}_{\ldoshounod} + \norm{P}_{\ldoslodoso}\leq C\set{\norm{\un_0}_{\ldoso} + \norm{\fn}_{\ldosldoso}},
\end{equation*}
for a constant $C>0$.
\end{theorem}
\begin{proof} It is sufficient to show that the conditions {H1}-{H5} in Theorem~\ref{main} hold true. We will only prove {H1} and {H4}, 
because the others conditions are immediate. 

In order to deduce {H1}, we need to prove that there exists $\beta>0$ such that
\begin{equation}\label{inf-sup-S}
\sup_{\vn\in\houno^d}\dfrac{-\int_\Omega {q\dive\vn}}{\norm{\vn}_{\houno^d}}\geq \beta\norm{q}_{\ldoso}\qquad\forall q\in\lodoso.
\end{equation}
In fact, let $q\in\lodoso$ with $q\neq 0$. Then, from Lemma~\ref{isomorfismo-div-grad}, there exists a unique $\wn\in V^\bot\subset\houno^d$
with $\wn\neq 0$ and
satisfying $\dive\wn=q$. Hence, by using Corollary~\ref{div-inverso}, it follows that
\[
\norm{\wn}_{\houno^d} \leq C\norm{q}_{\ldoso},
\]
for some positive constant $C$. Consequently,
\begin{equation*}
\sup_{\vn\in\houno^d}\dfrac{-\int_\Omega {q\dive\vn}}{\norm{\vn}_{\houno^d}}\geq
\dfrac{-\int_\Omega {q\dive(-\wn)}}{\norm{-\wn}_{\houno^d}} 
= \dfrac{\norm{q}_{\ldoso}^2}{\norm{\wn}_{\houno^d}}
\geq \dfrac{1}{C}\norm{q}_{\ldoso}\quad \forall q\in\lodoso,
\end{equation*}
{ which proves \eqref{inf-sup-S}}.

On the other hand, by noticing that
\[
a(\vn,\vn)=\nu\norm{\vn}_{\houno^d}^2\qquad\forall\vn\in\houno^d,
\]
we get \eqref{garding1} with $\gamma:=0$ and $\alpha:=\nu$, and therefore {H4} holds.
\end{proof}
%--------------------------------------------------------------------
\subsection{Application to the eddy current model}
In this section, we want to show other particular application for the theory, come{s} from mixed formulations {for} an electromagnetic problem so-called \textit{the eddy current model}.  
The eddy current model is obtained by dropping the displacement currents from Maxwell equations (see, \textit{e.g.} {\cite[chapter 10]{bossavit})}and it provides a reasonable approximation to the solution of the full Maxwell system in the low frequency range (see~\cite{AB}). Then{, the eddy current model} equations restricted to a domain $\Omega\subset \R^3$ {come to be:}
\begin{problem}
\label{problemaeddy}
Find $\bE:\Omega\times[0,T]\to\R$ and $\bH:\Omega\times[0,T]\to\R$ such that
\begin{align*}
\curl\bH=\bJ+\sigma\bE &\quad\text{in }\Omega\times[0,T],
\\
\frac{\partial}{\partial t}(\mu\bH)+\curl\bE=\cero &\quad\text{in }\Omega\times[0,T],
\\
\dive(\mu\bH)=0 &\quad\text{in }\Omega\times[0,T],
\\
\bH(\xn,0)=\bH_0 &\quad\text{in }\Omega.
\end{align*}
\end{problem}

The variable $\bE$ is the electric field, $\bH$ the magnetic field, $\bJ$ the current density, $\Hn_0$ is the initial magnetic field, $\mu$ the magnetic
permeability and $\sigma$ the electric conductivity. 
{Furthermore, the model consider the physical parameter $\varepsilon$, that is called  the electric permittivity.} 
The domain $\Omega$ has to be divided into two regions: the conductor $\oc$ and the insulator $\od=\Omega\setminus\overline{\oc}$. As usual, it was assumed that $\mu$, $\varepsilon$ and $\sigma$ are time-independent and piecewise smooth real valued functions satisfying:
\begin{align}
&\varepsilon_1\geq \varepsilon(\xn)\geq \varepsilon_0 > 0\quad\text{$\pct$ in $\oc$}
\ \ \quad\text{and}\quad\varepsilon(\xn)= \varepsilon_0 \quad \text{$\pct$ in $\od$},
\label{propepsilon}
\\
&\sigma_1\geq \sigma(\xn)\geq \sigma_0 > 0 \quad \text{$\pct$ in $\oc$} 
\quad\text{and}\quad\sigma(\xn)= 0 \quad \text{$\pct$ in $\od$},%\label{propsigma}
\\
&\mu_1\geq \mu(\xn)\geq \mu_0 > 0 \quad\text{$\pct$ in }\oc
\quad\text{and}\quad\mu(\xn)= \mu_0 \quad \text{ in $\od$}.\label{propmu}
\end{align}

The abstract theory developed in Section~\ref{continuo} can be used to study the continuous case of the mixed formulations proposed for the eddy 
current model in \cite{MS1,A1,A2,blrs:13}, but we will only focus in the formulations studied in \cite{A1} and \cite{blrs:13}. In both cases was deduced the 
formulation of the eddy current problem in terms of a time-primitive of the electric field, \textit{i.e.}, in terms of the unknown $\un$ given by
\begin{equation}\label{relauE}
\un(\xn,t):=\int_0^t \bE(\xn,s)ds.
\end{equation}
In \cite{A1}, {the eddy current problem} was studied for an internal conductor (\textit{i.e.}, 
the conductor satisfies $\partial\oc\cap \partial\Omega = \emptyset$), with the homogeneous Dirichlet boundary condition
\begin{align*}
\bE\times \nn=\cero \quad \text{on}\quad  [0,T]\times \partial\Omega.
\end{align*}
Moreover, they must impose the following conditions so that $\bE$ is uniquely determined
\begin{align*}
\dive(\varepsilon\bE)=0&\qquad\textrm{in }\od\times[0,T),
\\
\int_{\Sigma_i}\varepsilon\bE|_{\od}\cdot\bn \:=\:0&
\qquad \textrm{in }[0,T),\quad i=1,\dots,M_I,
\end{align*}
where $\Sigma_i$, $i=1,...,M_I,$ are the connected the components of $\Sigma:=\partial\oc$.

{On the other hand}, {i}n \cite{blrs:13} the eddy current model was analyzed with input currents intensities as source data, {they suppose two types of the conductor domain: internal conductors and inductors which go through the boundary of $\Omega$. In the second case, the boundary 
$\Gc:=\partial\oc\cap \partial\Omega$ %\ne \emptyset$
is not empty, and it is splitted as 
$\Gc=\GE\cup\GJ$, where $\GJ$ is the current input surface and $\GE$ is the current exit surface of the domain conductor. The connected components
of the boundary $\GJ$ are denoted by $\GJ^1,\GJ^2,\ldots, \GJ^N$ (see Figure~\ref{domain}, $N=3$.).
}

\begin{figure}[!ht]
	\begin{center}
		\includegraphics*[width=7cm]{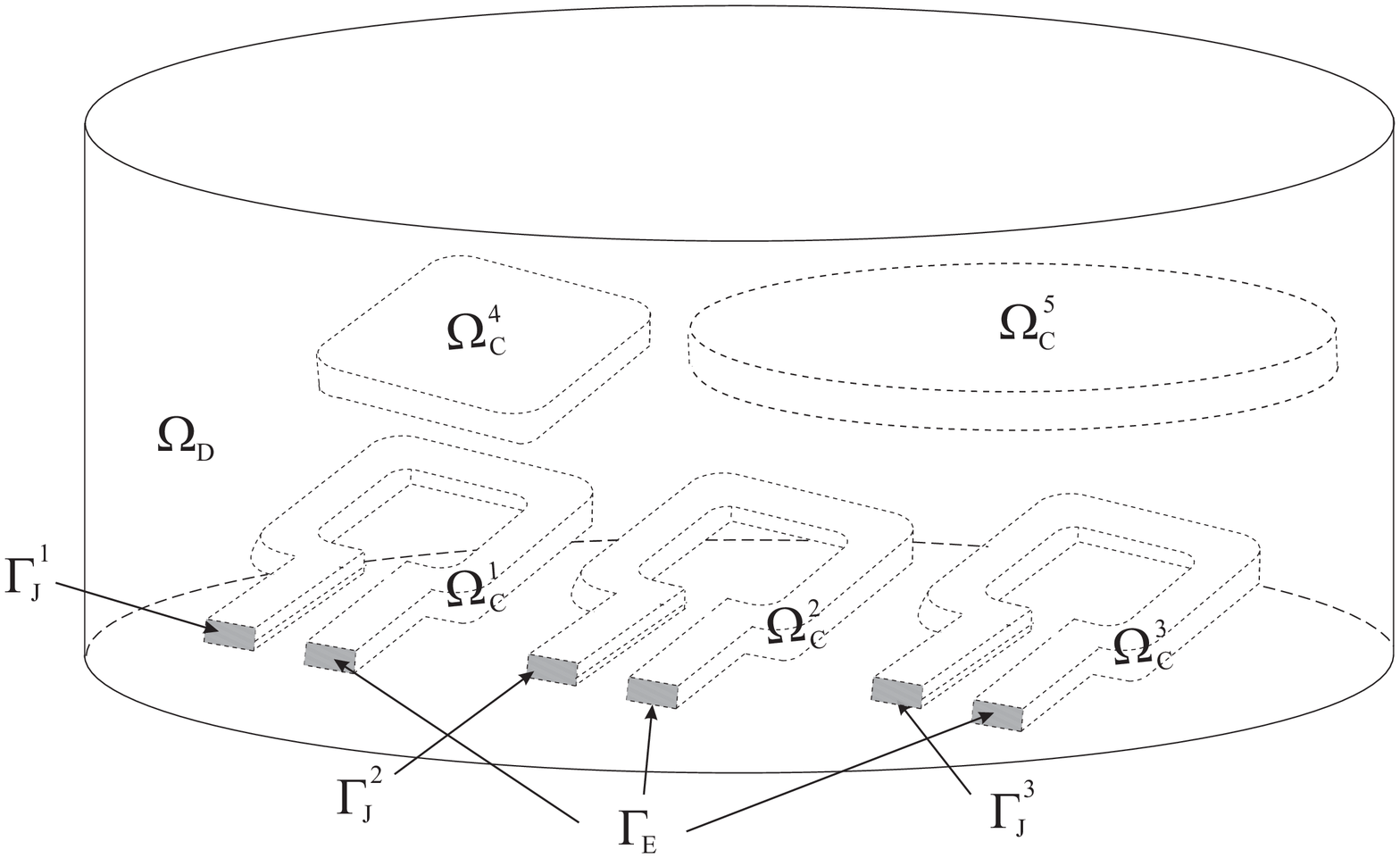}	
	\end{center}
	\caption{Sketch of the domain}
	\label{domain}
\end{figure}

The intensities of the input current are imposed by
\[
\int_{\GJ^n}\sigma\bE\cdot\bn=I_n \quad\text{in }[0,T],\quad n=1,\ldots,N,
\]
where $I_n$ is the current intensity through the surface $\GJ^n$, and the following boundary conditions was proposed
\begin{align*}
\bE\times\nn=\cero &\quad\text{on }[0,T]\times\GE,
\\
\bE\times\nn=\cero &\quad\text{on }[0,T]\times\GJ,
\\
\mu\bH\cdot\nn=0 &\quad\text{on }[0,T]\times\partial\Omega.
\end{align*}
In this case, $\bE$ is uniquely determined provided that
\begin{align*}
\dive(\varepsilon \bE)=0&\quad \text{in } [0,T]\times\od,
\\
\epsilon \bE|_{\od}\cdot\bn =g &\quad \text{on }[0,T]\times\Gd,
\\
\int_{\GI^k}\varepsilon\bE|_{\od}\cdot\bn=0, &\quad
k=2,\dots,M_{I}, \quad \mbox{in }[0,T],
\end{align*}
{where $\Gd:=\partial\Omega\cap\partial\od$; $\GI^k:=\partial\oc^k\cap\partial\od$, $k=1,\dots,M_{I},$ are the connected components of 
the interface boundary between the conducting and the insulator domain (see Figure~\ref{domain}, $M_{I}=5$) and $g$ is an additional data.}

In order to show the application of the  abstract theory in the { models studied in \cite{A1} and \cite{blrs:13},} we need to introduce
{ the following functional spaces:}
\begin{align*}
\hcurlo&:=\left\{\vn\in\ldoso^3:\: \curl\vn\in\ldoso^3\right\},\\
\hocurlo&:=\left\{\vn\in\hcurlo:\: \vn\times\nn=\cero \:\: \text{on}\: \partial
\Omega\right\},\\
\mathbf{H}(\curl^0;\Omega)&:=\left\{\vn\in\hcurlo:\: \curl\vn=\cero \:\: \text{in}\:
\Omega\right\}.
\end{align*}
Finally, given a subset $\Lambda\subset \partial\Omega$, we denote
%$\mathbf{H}_{0,\Lambda}(\curl;\Omega)$
%(respectively $\hodive$) denotes the space of functions 
%$\wn\in \hcurlo$ (respectively $\wn\in \hdive$) with vanishing tangential (normal) trace on $\partial{\Omega}$.
%By $\mathbf{H}(\curl^0;\Omega)$ (respectively $\mathbf{H}(\dive^0;\Omega)$) we denote the functions $\wn\in\hcurlo$ with
%vanish rotational (divergence) in $\Omega$. Given subset $\Lambda\subset \partial\Omega$ we denote by 
\begin{equation*}
 \mathbf{H}_{\Lambda}(\curl;\Omega):=\left\{\vn\in\hcurlo:\: \vn\times \nn=\cero \quad \text{on}\quad
\Lambda\right\}
\end{equation*}
and 
\begin{equation*}
 \mathbf{H}_{\Lambda}(\curl^0;\Omega):=\mathbf{H}(\curl^0;\Omega)\cap \mathbf{H}_{\Lambda}(\curl;\Omega).
\end{equation*}
Similarly, we consider analogous definitions for the spaces $\hdive$, $\hodive$, $\mathbf{H}_{\Lambda}(\dive;\Omega)$. Furthermore,
we introduce the space  
{
\[
\mathbf{H}_{\Lambda}(\dive^0_{\varepsilon};\Omega):=\left\{\wn\in\mathbf{H}_{\Lambda}(\dive;\Omega): \dive({\varepsilon\wn})=0 \right\}.
\]
}
{In the following, we refer to the problem studied in \cite{A1} as \textit{internal conductor model} and the problem studied 
in \cite{blrs:13} as \textit{the input current model}. 
}
%--------------------------------------------------------------------
\subsubsection{Internal conductor}
The eddy current model (see Problem~\ref{problemaeddy}) for the case of internal conductors, can be written in terms of the new variable $\un$ given by \eqref{relauE}, as the following system of equations (see \cite{A1}): 
\begin{problem}\label{fuerte_e-1}
Find $\un:\Omega\times [0,T)\rightarrow\Omega$ such that:
\begin{align*}
\sigma\partial_t\un + \curl\left(\frac{1}{\mu}\curl\un\right)=\curl\Hn_0 - \Jn
&\qquad\textrm{in }\Omega\times(0,T),
\\
\dive(\varepsilon\un)=0&\qquad\textrm{in }\od\times[0,T),
\\
\int_{\Sigma_i}\varepsilon\un\cdot\bn \:=\:0&
\qquad \textrm{in }[0,T),\quad i=1,\dots,M_{I},
\\
\un\times\bn=\cero&\qquad\textrm{in }\partial\Omega\times[0,T),
\\
\un(\cdot,0)=\cero&\qquad \text{in $\Omega$}.
\end{align*}
\end{problem}
We assume that $\Hn_0\in\hcurlo$ and $\Jn\in\L^2(0,T;\ldoso^3)$. Furthermore, we need to introduce the space 
\begin{equation*}
\mo:=\left\{\mu\in\hunod:\  \mu\vert_{\partial\Omega}=0\textrm{ y }
\mu\vert_{\Sigma_i}=C_i,\ i=1,\dots,M_{I}\right\},
\end{equation*}
where $C_i$, $i= 1,\dots,M_{I},$ are arbitrary constants. Now, if we denote for $t\in[0,T]$
\begin{equation}\label{f-eddy}
\left\langle \fn(t),\vn\right\rangle=\int_{\Omega}\fn(t)\cdot\vn := \int_{\Omega}\curl\Hn_0\cdot\vn - \int_{\Omega} \Jn(t)\cdot\vn\qquad \forall\vn\in\hocurlo,
\end{equation}
the weak formulation of the Problem \ref{fuerte_e-1} (see \cite{A1}) is
\begin{problem}\label{mix1-int}
 Find $\un\in\ldoshocurlo$ and $\lambda\in\ldosmo$ such that 
\begin{align*}
%\label{mix1-int}
&\dfrac{d}{dt}\left[\int_{\oc}\sigma\un(t)\cdot\vn + \int_{\od}\varepsilon\vn\cdot\nabla\lambda(t)\right]
+ \int_{\Omega}\frac{1}{\mu}\curl\un(t)\cdot\curl\vn =\int_{\Omega}\fn(t)\cdot\vn&&\forall \vn\in\hocurlo, \\
%\label{mix2-int}
& \int_{\od}\varepsilon\un(t)\cdot\nabla\mu = 0 &&\forall \mu\in\mo ,\\
%\label{mix3-int}
&\un(\cdot,0)=\cero\quad \textrm{in }\oc.
\end{align*}
\end{problem} 

In order to fit the Problem \ref{mix1-int} in the abstract theory studied in Section~\ref{continuo}, we have to define
\begin{equation}\label{defXYM}
X:=\hocurlo,\qquad Y:=\ldoso^3,\qquad M:=\mo,
\end{equation}
with their usual inner products. Then, we can easily deduce that these spaces satisfying the corresponding properties for the theory. Moreover, we need to define the operators  $A:X\to X'$, $R:Y\to Y'$ and the bilinear form $b:X\times M\to\R$ given by:
\begin{align}
\left\langle A\vn,\wn\right\rangle_{\X}&:=\int_\Omega\dfrac1\mu\curl\vn\cdot\curl\wn&&\forall\vn,\wn\in X,\label{defA}\\
\left\langle R\vn,\wn\right\rangle_{\Y}&:=\int_{\oc}\sigma\vn\cdot\wn&&\forall\vn,\wn\in Y,%\label{defR}
\\
b(\vn,\mu)&:= \int_{\od}\varepsilon\vn\cdot\nabla\mu&&\forall\vn\in X,\quad\forall\mu\in M.\label{defb}
\end{align}

We can now proceed to show that the Problem \ref{mix1-int} is a well-posed problem, but we first need to prove the following
result about the continuous kernel of b, which is given by
\begin{equation}\label{defV}
\begin{split}
V:=&\left\{\vn\in\hocurlo:\:\: b(\vn,\mu)=0\quad \forall\mu\in\mo\right\}\\
=&\left\{\vn\in\hocurlo:\:\: \dive(\varepsilon\vn)=0\textrm{ in }\od;\quad
\int_{\Sigma_i}\varepsilon\vn\cdot\bn \:=\:0,\: i=1,\dots,M_I\right\}.
\end{split}
\end{equation}

\begin{lemma}\label{lemmaz}
Let $\hat\bn$ be the unit outward normal on $\partial\Omega$. For any $\vn\in V$, there exists a unique $\zn\in\hcurlod$ satisfying
\begin{equation}\label{ztraza}
\zn\times\bn=\vn\times\bn\textrm{ on }\partial\oc,
\quad\zn\times\hat\bn=\cero\textrm{ on }\partial\Omega,
\end{equation}
and
\begin{equation}\label{zkernel}
b(\zn,\mu)=0\qquad\forall \mu\in\mo.
\end{equation}
Furthermore, if we denote $\vc:=\vn\vert_{\oc}$, then there exists a constant $C>0$ such that
\begin{equation*}
\begin{split}
\|\zn\|_{\hcurlod}
\leq C \|\vc\|_{\hcurloc}.
\end{split}
\end{equation*}
\end{lemma}
\begin{proof}
We start by considering the following space
\begin{align*}
\Vd:=&\left\{\vn\in\hocurlod:\:\: \dive(\varepsilon\vn)=0\textrm{ in }\od,\quad
\int_{\Sigma_i}\varepsilon\vn\cdot\bn \:=\:0,\: i=1,\dots,M_I \right\}\\[1.5ex]
=&\left\{\vn\in\hocurlod:\:\: b(\vn,\mu)=0\quad \forall\mu\in\mo\right\}.
\end{align*}
In this space, the seminorm $\vn\mapsto\norm{\curl\vn}_{\ldoso^3}$ is a norm, which is equivalent to the $\hcurlod$-norm, \textit{i.e.}, there exists a constant $C>0$ satisfying (see, \textit{e.g.}~Corollary~4.4~\cite{hiptmair})
\begin{equation}\label{poinc}
\norm{\vn}_{\hcurlod}\leq C  \norm{\curl \vn}_{\ldosod^3} \qquad \forall \vn \in \Vd.
\end{equation}

{On the other hand, let $\gt$ be the tangential trace operator on $\hcurlod$ and $\gt^{-1}$  its continuous right inverse; see, \textit{e.g.}, 
\cite[Theorem~4.6]{buffa1-2} and \cite[Theorem~4.1]{buffa}. Therefore, for all $\vn\in V$,  we can define $\vd:=\gt^{-1}(\vc\times\bn)$, where $\vc:=\vn\vert_{\oc}$. Hence, $\vd\times\hat{\bn}=\cero$ on $\partial\Omega$, and furthermore there exists $C>0$ such that
\begin{equation}\label{cont-vd}
\norm{\vd}_{\hcurlod}\leq C\ \norm{\vc}_{\hcurloc}\leq C\norm{\vn}_{\hcurlo}.
\end{equation}
}		
We now consider the following mixed problem, where $\vd$ is defined above: find $\wn\in \hocurlod$ and $\rho\in\mo$ such that
\begin{align}\label{defw1}
&\int_{\od}\curl\wn\cdot\curl\hat\vn
+ b(\hat\vn,\rho)
= -\int_{\od}\curl\vd\cdot\curl\hat\vn
&&\forall\hat\vn\in\hocurlod,\\
\label{defw2}
& b(\wn,\mu)=-b(\vd,\mu)
&&\forall \mu\in\mo.
\end{align}
The clasical Babuska-Brezzi Theory shows this problem is well posed. In fact, the ellipticity of bilinear form $(\wn,\vn)\mapsto \left(\curl\wn,\curl\vn\right)_{\ldosod^3}$ on the kernel of $b$ is given by \eqref{poinc} and the corresponding \textit{inf--sup} condition can be proved as follows: Let $\mu\in\mo$, then $\nabla\mu\in\hocurlod$ and we have
\begin{equation*}
\sup_{\hat\vn\in\hocurlod}\dfrac{b(\hat\vn,\mu)}
{\Vert\hat\vn\Vert_{\hcurlod}}\geq \dfrac{b({\nabla\mu},\mu)}
{\Vert{\nabla\mu}\Vert_{\hcurlod}}=
\dfrac{\int_{\od}\varepsilon\nabla\mu\cdot\nabla\mu}{\Vert\nabla\mu\Vert_{\ldosod^3}}
=\varepsilon_0\Vert\nabla\mu\Vert_{\ldosod^3}.
\end{equation*}
The \textit{inf--sup} condition follows by recalling that in the space the $\mo$ the standard norm in $\hunod$ is equivalent to its seminorm.
			
Next, if we define $\zn:=\wn+\vd$, then $\zn$ satisfies \eqref{ztraza} and \eqref{zkernel}.	Finally, since $\wn$ is the solution of the well-posed problem \eqref{defw1}--\eqref{defw2}, from \eqref{cont-vd} it follows that
\[
\|\zn\|_{\hcurlod}\leq \|\wn\|_{\hcurlod} + \|\vd\|_{\hcurlod}
\leq C \|\vd\|_{\hcurlod} 
\leq C\|\vc\|_{\hcurloc}.
\]
Thus we conclude the proof.
\end{proof}
	
\begin{theorem}\label{well-eddy}
The Problem~\ref{mix1-int} has a unique solution $(\un,\lambda)\in\ldoshocurlo\times\H^1(0,T;\M(\od))$ and there exists a constant $C>0$ such that
\[
\|\un\|^2_{\L^2(0,T;\hcurlo)} + 
\|\lambda\|^2_{\L^2(0,T;\hunod)}
\leq C \left(\|\curl\Hn_0\|_{\ldoso^3}^2 
+\|\Jn\|^2_{\L^2(0,T;\ldoso^3)}\right).
\]
\end{theorem}

\begin{proof}
We first notice that if we use the notation given by \eqref{defXYM}--\eqref{defb}, the eddy current formulation the Problem \ref{mix1-int} is a particular case of the mixed parabolic the Problem~\ref{PC1}, where $\fn$ is given by \eqref{f-eddy}, $u_0=\cero$ in $\oc$ and $g=0$. Then, we need to verify the hypothesis of Theorem~\ref{main}. 

It is a simple matter to see the functional spaces and embeddings concerning to the eddy current problem satisfy the conditions of theorem. Consequently, it remains to verify the hypotheses H1--H6 from Section~\ref{continuo}. We only present the proof of H1 and H4, because  the proof of the rest of these properties is straightforward.
\begin{itemize}
\item[H1.] Let $\mu\in\mo$. Thus $\nabla\mu\in\hocurlod$ and furthermore, if we denote $\widetilde{\nabla\mu}$ 
the extension by zero to the whole $\Omega$ of $\nabla\mu$, then $\widetilde{\nabla\mu}\in\hocurlo$.
Consequently, by using \eqref{propepsilon}, for all $\mu\in\mo$ we obtain 
\[
\sup_{\vn\in\hocurlo}\dfrac{b(\vn,\mu)}
{\Vert\vn\Vert_{\hcurlo}}\geq \dfrac{b(\widetilde{\nabla\mu},\mu)}
{\Vert\widetilde{\nabla\mu}\Vert_{\hcurlo}}=
\dfrac{\int_{\od}\varepsilon\nabla\mu\cdot\nabla\mu}{\Vert\nabla\mu\Vert_{\ldosod^3}}
=\varepsilon_0\Vert\nabla\mu\Vert_{\ldosod^3}.
\]
Finally, by recalling that in the space the $\mo$ the standard norm in $\hunod$ is equivalent to its seminorm, we conclude that $b$ satisfies the continuous \textit{inf--sup} condition.
\item[H4.] We need to prove that there exist positive constants $\gamma$ and $\alpha$ such that
\begin{equation}\label{garding2}
\int_{\Omega}\dfrac{1}{\mu}\vert\curl\vn\vert^2 + \gamma\int_{\oc}\sigma\vert\vn\vert^2 
\geq \alpha\norm{\vn}_{\hcurlo}^2\qquad\forall\vn\in V,
\end{equation}
where $V$ is the continuous kernel of $b$ (see \eqref{defV}). 
		
Let $\vn\in V$, $\vc:=\vn\vert_{\oc}$ and $\zn\in\hcurlod$ given by Lemma~\ref{lemmaz}. We define $\E\vn$ as follows
\begin{equation*}%\label{defE}
\E\vn := \left\{
\begin{array}{lll}
\vc,&\ \textrm{ in }\oc,\\
\zn,&\ \textrm{ in }\od.
\end{array}
\right.
\end{equation*}
Then $\E\vn\in\hocurlo$, $b(\E\vn,\mu)=b(\zn,\mu)=0$ for all $\mu\in\mo$ and therefore $\E\vn\in V$. Moreover,
\begin{equation}\label{Econtinuo}
\norm{\E\vn}_{\hcurlo} = \norm{\vc}_{\hcurloc} + \norm{\zn}_{\hcurlod} \leq C\norm{\vc}_{\hcurloc}.
\end{equation}
Thus, if we define $\widetilde\wn:=(\vn-\E\vn)$ then 
\[
\widetilde\wn\in V, \quad\widetilde\wn=\cero\textrm{ in } \oc\quad 
\textrm{and}\quad \widetilde\wn \vert_{\od}\in \Vd.
\]
Hence, from \eqref{poinc} and \eqref{Econtinuo}, it follows that
\begin{multline*}
\norm{\widetilde\wn}_{\hcurlo} = \norm{\widetilde\wn\vert_{\od}}_{\hcurlod} 
\leq C \norm{\left(\curl\widetilde\wn\right)\vert_{\od}}_{\ldosod^3}\\
\leq C \left\{\norm{\curl\E\vn}_{\ldosod^3} + \norm{\curl\vn}_{\ldosod^3}\right\}
\leq C \left\{\norm{\vc}_{\hcurloc} + \norm{\curl\vn}_{\ldosod^3}\right\}.
\end{multline*}
Consequently,
\[
\begin{split}
\norm{\vn}_{\hcurlo}^2 &= \norm{\E\vn + \widetilde\wn}_{\hcurlo}^2
\leq 2\left\{\norm{\E\vn}_{\hcurlo}^2 + \norm{\widetilde\wn}_{\hcurlo}^2\right\}\\
&\leq C\left\{\norm{\vn}_{\hcurloc}^2 + \norm{\curl\vn}_{\ldosod^3}^2\right\}.
\end{split}
\]
Finally, using this previous inequality and recalling \eqref{propepsilon}-\eqref{propmu}, we conclude \eqref{garding2}.
\end{itemize}
\end{proof}

\begin{remark}
The Lagrange multiplier $\lambda$ of the Problem \ref{mix1-int} vanishes identically   \cite[Lemma~4.5]{A1}.
\end{remark}
%-------------------------------------------------------------------
\subsubsection{Input currents}
{The analysis of the well-possednes of the problem with input currents was performed in \cite{blrs:13} by considering an equivalent problem. 
More precisely, it was necessary to use a result previously showed in \cite{BLRS1} for the eddy current problem in terms of the magnetic field $\bH$ and subsequently, 
it was required to extend the electric field $\bE|_{\oc}$ to the insulator domain $\od$. 
In this subsection, we will directly obtain the well-posedness of the problem without the need to resort to an equivalent problem, by showing 
that this formulation is a particular case of the abstract framework studied in Section~\ref{continuo}.

The corresponding strong problem (see \cite{blrs:13}) is
\begin{problem}\label{fuerte_puertos}
Find $\un:\Omega\times [0,T)\rightarrow\Omega$ such that:
\begin{align*}
\sigma\partial_t\un + \curl\left(\frac{1}{\mu}\curl\un\right)=\curl\Hn_0
&\qquad\textrm{in }\Omega\times[0,T],
\\
\curl\un\cdot\nn=0&\qquad\textrm{on }\partial\Omega\times[0,T],
\\
\un\times\nn=\cero &\qquad\textrm{on }\Gc\times[0,T],
\\
\dive(\varepsilon\un)=0&\qquad\textrm{in }\od\times[0,T],
\\
\varepsilon\un(t)\cdot\nn=\int_0^tg(s)ds&\qquad\textrm{on }\Gd\quad t\in[0,T]
\\
\langle\varepsilon\un\cdot\nn,1 \rangle_{\Gamma_I^k} \:=\:0&
\qquad k=2,\dots,M_{I},\qquad \textrm{in }[0,T]\quad 
\\
\langle \varepsilon\un(t)\cdot\nn,1\rangle_{\Gamma_j^n}=\int_0^t I_n(s)ds, &\qquad n=1,\ldots,N,\quad t\in[0,T],
\\
\un(\cdot,0)=\cero&\qquad \text{in $\Omega$}
\end{align*}
\end{problem}

We introduce the following Hilbert spaces 
\begin{align*}
\X&:=\set{\wn\in\hcurlo:\ \wn\times\nn=\cero\textrm{ on }\Gc,\ \curl\wn\cdot\nn=\cero\textrm{ on }\partial\Omega},\\
\M&:=\set{\varphi\in\hunod:\ \varphi\vert_{\GI^1}=0;\ \varphi\vert_{\GI^k}=C_{I}, k=2,\ldots,M_{I}}.
\end{align*}
with their usual norms in $\hcurlo$ and $\H^1(\Omega)$ respectively.

Given $g\in\L^2(0,T;\L^2(\Gamma))$, $I_n\in\H^2(0,T),\,\,n=1,\ldots,N$ and $\Hn_0$ { the initial magnetic condition}, if we denote for $t\in[0,T]$
\begin{equation*}
\left\langle \fn(t),\vn\right\rangle=\int_{\Omega}\fn(t)\cdot\vn :=\sum_{n=1}^N L_n(\vn)(I_n(t)-I_n(0))
+\int_{\Omega}\curl \Hn_0\cdot\vn
\quad \forall \vn\in \X
\end{equation*} 
then, the weak formulation of Problem \ref{fuerte_puertos} (see \cite{blrs:13}){ can read as follows :
}
\begin{problem}\label{mix1-puertos}
Find $\un\in\L^2(0,T;\X)$ and $\lambda\in\L^2(0,T;M)$ such that 
\begin{align*}
&\dfrac{d}{dt}\left[\int_{\oc}\sigma\un(t)\cdot\vn + \int_{\od}\varepsilon\vn\cdot\nabla\lambda(t)\right]
+ \int_{\Omega}\frac{1}{\mu}\curl\un(t)\cdot\curl\vn=\langle \fn(t),\vn\rangle
&& \forall \vn\in \X,
 \\
& \int_{\od}\varepsilon\un(t)\cdot\nabla\mu = \int_{\Gd}\left(\int_0^tg(s)\ds\right)\mu&&\forall \mu\in\M,
\\
&\un(\cdot,0)=\cero\quad \textrm{in }\oc,
\end{align*}
\end{problem} 

We have introduced the time primitive for the multiplier:
\begin{align*}
\lambda(\xn,t)=\int_0^t\xi(\xn,s)\ds\quad \xn\in \od\,,\quad t\in [0,T].
\end{align*}
Our next goal is to { fit}  Problem \ref{mix1-puertos} in the abstract theory  {of} Section {\ref{continuo}}, To this aim,
we will use the operators  $A:\X\to\X$, $R:\L^2(\Omega)^3\to\L^2(\Omega)^3$ and $b:\X\times\M\to\R$
{defined of analogous way to the operators given in  \eqref{defA}-\eqref{defb}.} 

We proceed to show that the Problem \ref{mix1-puertos} is a well-posed problem, then we now introduce the continuous kernel of $b$, which is given by
\begin{align*}
\VV=\left\{\wn\in \X:\,\,b(\wn,\mu)=0\,\,\,\forall\mu\in\M\right\}.
\end{align*}
It is easy show that 
\begin{align*}
\wn\in\VV \quad\Leftrightarrow\quad
\left\{
\begin{array}{ll}
&\wn\in\hcurlo,
\\
&\wn\times\nn=\cero\textrm{ on }\Gc,
\\
&\curl\wn\cdot\nn=\cero\textrm{ on }\partial\Omega,
\\
&\dive(\varepsilon\wn)=0\textrm{ in }\od,
\\
&\varepsilon\wn\cdot\nn=0\textrm{ on }\Gd,
\\
&\int_{\GI^k}\varepsilon\wn\cdot\nn = 0,\quad k=2,\ldots,M_{I}.
\end{array}
\right.
\end{align*}
To continue with the well-posedness of the Problem \ref{mix1-puertos},  it is necessary to introduce the following functional spaces
\begin{align*}
\HH&:=\HGIcurlood\cap\HGddivood,\\
\HHGIcurlod&
:=\set{\wn\in\HGIcurlod:\ \curl\wn\cdot\nn=0 \textrm{ on }\Gd},\\
 \HHdivod & :=\set{\wn\in\hdivod:\ \wn\cdot\nn\vert_{\Gd}\in\LdosGd},
\end{align*}
{equipped with their usual norms.} 
Furthermore, let us denote $\VVd$ as the kernel of $b$ on the space $\HHGIcurlod\cap\HH^\bot$, namely
\begin{align*}
\VVd:= 
\set{\wn\in\HHGIcurlod\cap\HH^\bot:\ b(\wn,\varphi)=0\quad \forall\varphi\in\M}.
\end{align*}
The following result {is directly obtained from} \cite[Proposition 7.1]{FG}.
\begin{lemma}\label{lemmaFG}
There exist a constant $C>0$ such that
\begin{equation*}
\norm{\vn}_{\ldosod^3} \leq C\left\{\norm{\curl\vn}_{\ldosod^3}
+ \norm{\dive(\varepsilon\vn)}_{\ldosod} + \norm{\varepsilon\vn\cdot\nn}_{\LdosGd}
\right\} 
\end{equation*}
for all  $\vn\in \HHGIcurlod\cap\HHdivod\cap\HH^{\bot}$. In particular,
\begin{equation*}
\norm{\vn}_{\ldosod^3} \leq C\norm{\curl\vn}_{\ldosod^3}\qquad\forall\vn\in\VVd.
\end{equation*}
\end{lemma}

\begin{lemma}\label{lemaE}
The lineal mapping $\E:\HGccurloc\to\VV$ 
characterized by 
\begin{equation}\label{caracterizacionE}
\begin{split}
&(\E\vc)\vert_{\oc}=\vc\qquad\forall\vc\in\HGccurloc;\\
&\int_{\od}(\curl\E\vc)\cdot\curl\wnd = 0 \qquad\forall \vc\in\HGccurloc
\quad \forall\wnd\in\VVd.
\end{split}
\end{equation}
is well defined and bounded.
\end{lemma}
\begin{proof}
Let us denote by $\gtc:\hcurloc\to\rangoGtoc $ and $\gtd:\hcurlod\to\rangoGtod$ the tangential traces on $\hcurloc$ and $\hcurlod$, respectively. We know that in both cases the operators are continuous, surjectives and continuous right inverse  \cite[Theorem 4.1]{buffa}. It follows that linear operator
$\etan:\HGccurloc\to\rangoGtod$ given by 
\begin{equation*}
\etan(\vc):=
\left\{
\begin{array}{ll}
\gtc(\vc)\vert_{\GI}&\textrm{on }\GI,\\
\cero,&\textrm{on }\Gd,
\end{array}
\right.
\end{equation*}
is well defined. Moreover, we have
\begin{equation*}
\norm{\etan(\vc)}_{\rangoGtod} = \norm{\gtc(\vc)}_{\rangoGtoc}\leq C_1\norm{\vc}_{\hcurloc}
\qquad\forall\vc\in\HGccurloc.
\end{equation*}
Futhermore, we define the linear and continuous operator $\LL:\HGccurloc\to\hcurlod$ as
\begin{equation*}
\LL(\vc):=(\gtd)^{-1}(\etan(\vc))\qquad\forall \vc\in\HGccurloc,
\end{equation*}
which satisfies 
\begin{equation*}
\LL(\vc)\in\HGdcurlod,\quad \LL(\vc)\vert_{\GI}\times\nn=\vc\vert_{\GI}\times\nn, \quad \forall \vc\in\HGccurloc.
\end{equation*}

Given $\vc\in\HGccurloc$, we consider the following mixed problem: 
\begin{problem}\label{mixtoinput}
\noindent Find $\znd\in\HHGIcurlod\cap\HH^\bot$ and $\rho\in\M$ such that
\begin{equation*}
\begin{split}
\int_{\od}\curl\znd\cdot\curl\wnd + b(\wnd,\rho) &= - \int_{\od}\curl(\LL(\vc))\cdot\curl\wnd\qquad
\forall\wn\in \HHGIcurlod\cap\HH^\bot\\
b(\znd,\mu) &= -b(\LL(\vc),\mu)\qquad\forall\mu\in\M.
\end{split}
\end{equation*}
\end{problem}
The well-possedness of the Problema \ref{mixtoinput} is in virtue of Babu$\check{\text{s}}$ka-Brezzi. In fact, 
Then by using the Lemma \ref{lemmaFG} we obtain that the 
bilinear form
\[
(\vnd,\wnd)\mapsto\int_{\od}\curl\vnd\cdot\curl\wnd,
\]
is coercive on $\Vd$. Furthermore, the bilinear form $b$ satisfies the \textit{inf--sup}condition{:}
\[
\sup_{\vnd\in\HHGIcurlod\cap\HH^\bot}\dfrac{b(\vnd,\mu)}{\norm{\vnd}_{\hcurlod}}
\geq \dfrac{b(\grad\mu,\mu)}{\norm{\grad\mu}_{\hcurlod}}
= \varepsilon_{0}\norm{\grad\mu}_{\ldosod^3},
\quad \forall\mu\in\M
\]
where, we have used that $\grad\mu\in\HHGIcurlod\cap\HH^\bot$.
Consequently, there exists a unique solution for the Problem \ref{mixtoinput},  namely 
\[
\znd\in\HHGIcurlod\cap\HH^\bot\quad\text{and}\quad
\rho\in\M.
\] 
{Thanks to the} well-posedness of the mixed Problem { \ref{mixtoinput}}  and the continuity of the operator $\LL$, we obtain
\[
\norm{\znd}_{\hcurlod}\leq C_2\norm{\vc}_{\hcurloc}\qquad\forall\vc\in\HGccurloc
\]
and we can notice that there holds
\[
\vc\vert_{\GI}\times\nn = \LL(\vc)\vert_{\GI}\times\nn\quad \text{and}\quad \znd\vert_{\GI}\times\nn=\cero \qquad\forall\vc\in\HGccurloc.
\]
Finally, we define 
\begin{equation*}
\E\vc:=\left\{
\begin{array}{ll}
\vc,&\textrm{in }\oc\\
\znd+\LL\vc,&\textrm{in }\od,
\end{array}
\right.
\end{equation*} 
which satisfies  \eqref{caracterizacionE}.
\end{proof}
\begin{theorem}\label{well-eddy2}
The Problem~\ref{mix1-puertos} has a unique solution $(\un,\lambda)\in\L(0,T;\X)\times {\H^1}(0,T;\M)$ and there exists a constant $C>0$ such that
\[
\|\un\|^2_{\L^2(0,T;\hcurlo)} + 
\|\lambda\|^2_{\L^2(0,T;\hunod)}
\leq C \left(
\sum_{n=1}^N\|I_n\|^2_{\H^1(0,T)}+\|\curl\Hn_0\|_{\ldoso^3}^2+\norm{g}_{\L^2(\Gd)}^2\right).
\]
\end{theorem}

\begin{proof}
Our goal is to show that the Problem \ref{mixtoinput} satisfies hypotheses given in Section~\ref{continuo}. We only verify the conditions H1 and H4, the proof of the rest of these properties is straightforward. 
\begin{itemize}
\item[H1.] {Let $\mu\in\M$. If we denote $\widetilde{\mu}$ 
the extension of $\mu$ to $\Omega_c$, given by $\widetilde{\mu}|_{\oc^k}=\mu|_{\GI^k}$, $k=2,\ldots,M_{I}$.}
Then $\nabla\widetilde{\mu}\in \X$ and
\[
\sup_{\vn\in\hocurlo}\dfrac{b(\vn,\mu)}
{\Vert\vn\Vert_{\hcurlo}}\geq \dfrac{b(\widetilde{\nabla\mu},\mu)}
{\Vert\widetilde{\nabla\mu}\Vert_{\hcurlo}}=
\dfrac{\int_{\od}\varepsilon\nabla\mu\cdot\nabla\mu}{\Vert\nabla\mu\Vert_{\ldosod^3}}
=\varepsilon_0\Vert\nabla\mu\Vert_{\ldosod^3}.
\]
Finally, thanks to the  Poincar\'e inequality { in $\M$}, it follows that the \textit{inf--sup} condition hold{s} true.
\item[H4.] We need to prove that there exist positive constants $\gamma$ and $\alpha$ such that
\begin{equation}\label{garding3}
\int_{\Omega}\dfrac{1}{\mu}\vert\curl\vn\vert^2 + \gamma\int_{\oc}\sigma\vert\vn\vert^2 
\geq \alpha\norm{\vn}_{\hcurlo}^2\qquad\forall\vn\in \VV,
\end{equation}
where $\VV$ is the continuous kernel of $b$ (see \eqref{defV}). 
		
Let $\vn\in \VV$, $\vc:=\vn\vert_{\oc}$ and $\E\vn$ as in Lemma \ref{lemaE}. Then $\E\vn\in\HGccurloc$, $b(\E\vn,\mu)=0$ for all $\mu\in\M$ and therefore $\E\vn\in \VV$. Thus, if we denote $\widetilde\wn:=(\vn-\E\vn)$ there holds
\[
\widetilde\wn\in \VV, \quad\widetilde\wn=\cero\textrm{ in } \oc\quad 
\textrm{and}\quad \widetilde\wn \vert_{\od}\in \VVd.
\]
Finally, by {proceeding} as in the end of the proof of Theorem \ref{well-eddy}, we conclude the G\r{a}rding inequality {\eqref{garding3}}. 
\end{itemize}
\end{proof}
%---------------------------------------------------------------------
\section{Conclusions}\label{conclusiones}

\begin{itemize}
\item It was possible to join the known theory for (stationary) mixed problems with the theory for linear degenerate parabolic equations.
\item There were obtained sufficient conditions about the involved functional spaces and operators to guarantee the existence and uniqueness of solution for a family of linear mixed degenerate problem.
\item It was possible to apply the developed theory for linear mixed degenerate parabolic equations at least to two different {kinds} of problems: \textit{the classical {time-dependent} Stokes problem}, which is a fundamental model of viscous flow in fluid dynamics; and \textit{the eddy current problem} that is an electromagnetic model that provides a reasonable approximation to the solution of the full Maxwell system in the {low-frequency} range. 
\item The application to the eddy current model includes the case of interior conductors and the case of input currents. The theory allows us to obtain directly the well-posedness 
of the eddy current model for input currents, without the need to resort to another equivalent system of equations. 
\end{itemize}

%---------------------------------------------------------------------
\section*{Acknowledgments}

This work was partially supported by Colciencias through the $727$ call, by University of Cauca  through VRI project ID $5243$ and by Universidad Nacional de Colombia through Hermes project $46332$.
%---------------------------------------------------------------------

%------------------------------------------------------------------
\end{document}